\documentclass[12pt]{amsart}
\usepackage{amsmath,amsthm,amsfonts,amssymb,latexsym,verbatim,bbm}
\input xy
\xyoption{all}
\usepackage{hyperref,multirow}
\usepackage[shortlabels]{enumitem}
\usepackage{young}
\usepackage{tikz, ifthen}

\usepackage[backend=bibtex]{biblatex}
\usepackage[nomessages]{fp}
\usepackage{mathtools}
\usetikzlibrary{calc}
\usetikzlibrary{chains}
\usepackage{booktabs,caption}

\newcounter{x}
\newcounter{y}
\newcounter{z}

\newcommand*\cubecolors[1]{%
  \ifcase#1\relax
  \or\colorlet{cubecolor}{cyan}%
  \or\colorlet{cubecolor}{green}%
  \or\colorlet{cubecolor}{yellow}%
  \or\colorlet{cubecolor}{pink}%
  \or\colorlet{cubecolor}{purple}%
  \or\colorlet{cubecolor}{blue}%
  \else
    \colorlet{cubecolor}{white}%
  \fi
}
\newcommand\yaxis{180}
\newcommand\zaxis{-27}
\newcommand\xaxis{90}

\newcommand\topside[3]{
  \fill[fill=cubecolor, draw=black,shift={(\xaxis:#1)},shift={(\yaxis:#2)},
  shift={(\zaxis:#3)}] (0,0) -- (1,0) -- (0.5,0.25) --(-0.5,0.25)--(0,0);
}

\newcommand\leftside[3]{
  \fill[fill=cubecolor, draw=black,shift={(\xaxis:#1)},shift={(\yaxis:#2)},
  shift={(\zaxis:#3)}] (0,0) -- (0,-1) -- (-0.5,-0.75) --(-0.5,0.25)--(0,0);
}

\newcommand\rightside[3]{
  \fill[fill=cubecolor, draw=black,shift={(\xaxis:#1)},shift={(\yaxis:#2)},
  shift={(\zaxis:#3)}] (0,0) -- (1,0) -- (1,-1) --(0,-1)--(0,0);
}

\newcommand\cube[3]{
  \topside{#1}{#2}{#3} \leftside{#1}{#2}{#3} \rightside{#1}{#2}{#3}
}

\newcommand\planepartition[2]{
 \setcounter{x}{0}
 \foreach \a in {#2} {
    \addtocounter{x}{1}
    \setcounter{y}{-1}
    \cubecolors{\value{x}}
    \foreach \b in \a {
      \addtocounter{y}{1}
      \setcounter{z}{-1}
      \foreach \c in {0,...,\b} {
        \addtocounter{z}{1}
      \ifthenelse{\c=0}{\setcounter{z}{-1},\addtocounter{y}{0}}{
        \FPeval{\newz}{clip(0.55*\the\value{z})}
        \cube{\value{x}}{\value{y}}{\newz};
        \FPeval{\result}{clip(#1-\the\value{y}+2*\the\value{z})}
        \draw[draw=black,shift={(\xaxis:\value{x})},shift={(\yaxis:\value{y})},
  shift={(\zaxis:\newz)}] (0.5,-0.5) node {\textsf{\hspace{0.01cm}}};}
      }
    }
  }
}

\newcommand\planepartitionD[2]{
 \setcounter{x}{0}
 \foreach \a in {#2} {
    \addtocounter{x}{1}
    \setcounter{y}{-1}
    \cubecolors{\value{x}}
    \foreach \b in \a {
      \addtocounter{y}{1}
      \setcounter{z}{-1}
      \foreach \c in {0,...,\b} {
        \addtocounter{z}{1}
      \ifthenelse{\c=0}{\setcounter{z}{-1},\addtocounter{y}{0}}{
        \FPeval{\newz}{clip(0.55*\the\value{z})}
        \cube{\value{x}}{\value{y}}{\newz};
        \FPeval{\result}{clip(#1-\the\value{y}-2*\the\value{z})}
        \draw[draw=black,shift={(\xaxis:\value{x})},shift={(\yaxis:\value{y})},
  shift={(\zaxis:\newz)}] (0.5,-0.5) node {\textsf{\result}};}
      }
    }
  }
}

\tikzset{node distance=2em, ch/.style={circle,draw,on chain,inner sep=2pt},chj/.style={ch,join}, line width=1pt,baseline=-1ex}

\allowdisplaybreaks[2] \textwidth15.1cm \textheight22cm \headheight12pt \oddsidemargin.4cm
\theoremstyle{plain}
\newtheorem{theorem}{Theorem}[section]

\newtheorem{lemma}[theorem]{Lemma}

\newtheorem{example}[theorem]{Example}
\newtheorem{definition}[theorem]{Definition}
\newtheorem{defn}[theorem]{Definition}

\newtheorem{corollary}[theorem]{Corollary}

\newcommand{\R}{\mathbb{R}}
\newcommand{\C}{\mathbb{C}}
\newcommand{\Z}{\mathbb{Z}}
\newcommand{\N}{\mathbb{N}}
\newcommand{\Q}{\mathbb{Q}}
\newcommand{\K}{\mathbb{K}}
\newcommand{\T}{\mathbb{T}}

\newcommand{\mcB}{\mathcal{B}}

\newcommand{\mcD}{\mathcal{D}}

\DeclareMathOperator{\adj}{adj}

\DeclareMathOperator{\U}{U}

\DeclareMathOperator{\supp}{S}

\begin{document}
\title[Enumerating Staircase Diagrams over type $E$ Dynkin Diagrams]{Enumerating Staircase Diagrams and Smooth Schubert Varieties over type $E$ Dynkin Diagrams}

\author{Andean Medjedovic}
\email{a2medjed@uwaterloo.ca}

\author{William Slofstra}
\email{weslofst@uwaterloo.ca}

\begin{abstract}
    We enumerate the number of staircase diagrams over classically finite $E$-type Dynkin diagrams,
    extending the work of Richmond and Slofstra (Staircase Diagrams and Enumeration of smooth Schubert varieties) and
    completing the enumeration of staircase diagrams over finite type Dynkin diagrams. The staircase diagrams are in bijection to smooth
    and rationally smooth Schubert varieties over $E$-type thereby giving an enumeration of these varieties.
\end{abstract}
\maketitle
\section{Introduction}

\subsection{Staircase Diagrams}
The goal of this paper is to complete the enumeration done by Slofstra and Richmond in \cite{RS15}. It was known from \cite{RS15} and \cite{BP05} that it is possible to
enumerate rationally smooth Schubert varieties through what is known as a Billey-Postnikov decomposition. There is a bijection, derived in one of the authors previous works, between staircase diagrams, and these smooth varieties.\\

Informally, a staircase diagram is a construction of ``blocks'' over some underlying graph. Each block is some subset of the vertices of the graph, with some ordering of the blocks. We also require each block to ``see'' both up and down. This corresponds to a minimal and maximal in the poset for every given block. To illustrate the idea, what follows is a staircase diagram on the line graph of length $9$, another diagram on a star graph with $3$ leaves and a non-example.\\

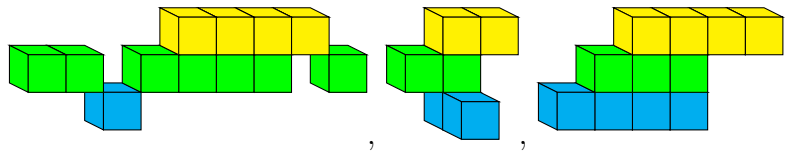
\begin{center}
\begin{tikzpicture}[scale=0.5]
            \planepartition{}{ {0,0,0,0,0,0,0,0,1,0},{0,0,1,0,1,1,1,1,0,1,1},{0,0,0,1,1,1,1}}.
        \end{tikzpicture},
\begin{tikzpicture}[scale=0.5]
            \planepartition{}{ {0,0,0,0,0,2,0,0,0,0},{0,0,0,0,0,1,1,0,0,0,0},{0,0,0,0,1,1,0}}.
        \end{tikzpicture},
\begin{tikzpicture}[scale=0.5]
            \planepartition{}{ {0,0,0,0,0,1,1,1,1,0},{0,0,0,0,0,1,1,1,0,0,0},{0,0,0,1,1,1,1}}.
        \end{tikzpicture}

\end{center}

Let $E_{Disc.}(x) = \sum_{n=1}^{\infty} e_n x^n$, where $e_n$ is the number of smooth Schubert varieties of finite classical type $E_n$. The main result of the paper is
\begin{theorem}
	\[
		E_{Disc.}(x) = \frac{P_{Disc.}(x)+Q_{Disc.}(x)\sqrt{1-4x}}{R_{Disc.}(x)}
	.\]
Where \tiny
$$P_{Disc.}(x) = 1536 x^{16}-9586 x^{15}+20762 x^{14}-19750 x^{13}+10942 x^{12}-15139 x^{11}+27760 x^{10}-28954 x^9$$
$$+16898 x^8-4690 x^7-173 x^6+689 x^5-454 x^4+238 x^3-73 x^2+8 x$$
$$Q_{Disc.}(x) = 3224 x^{16}-13230 x^{15}+21016 x^{14}-15930 x^{13}+4800 x^{12}+3759 x^{11}-10616 x^{10}+13958 x^9$$
$$-10482 x^8+4200 x^7-695 x^6-95 x^5+168 x^4-140 x^3+57 x^2-8 x$$
\normalsize
and
$$R_{Disc.}(x) = (-1 + x)^4 (-1 + 6 x - 8 x^2 + 4 x^3).$$
\end{theorem}

The growth rate of the coefficients of a generating series is determined by the singularity of smallest modulus. This result can be found in \cite{FS09}.
In this case it is one of the roots of $\left( -1 + 6x - 8x^2 +4x^3 \right) $,
\[
\alpha := \frac{1}{6}\left( 4 - \sqrt[3]{17+3\sqrt{33}} + \sqrt[3]{-17+3\sqrt{33} } \right) \approx 0.228155
.\]

Then $e_n \sim \frac{1}{\alpha^{n+1}}$.
This agree with the growth rate of all enumerations over classical finite type Dynkin diagrams \cite{RS15}.

\subsection{Outline}
We begin by introducing formally the notion of a staircase diagram and definitions corresponding to it. We then formalize tools to dissect these diagrams
into parts that are easier to enumerate combinatorially, giving plenty of examples along they way to build intuition. We prove a few lemmas relating to the possible endings staircase diagrams can have and the form these endings take. We define sight-sets of a partial staircase diagram which encodes how close to being a valid staircase diagram it is. In section $4$ we work on the connected $E$-type case using generating function techniques related to dyck paths and the theory so far developed. In the penultimate section we deal with the more complicated disconnected case using the now complete classification for connected staircase diagrams. Finally, we leave with an explanation of how staircase diagrams relate to smooth Schubert varieties and the Billey-Postnikov decomposition \cite{BP05}.
\subsection{Acknowledgements}

We would like to thank NSERC for providing the funding for this research,
and the members of IQC for helpful discussions. Finally we would like to thank
Jang Soo Kim, for TikZ plane partition code which was modified to make the staircase diagrams within the paper.

\section{Preliminary Definitions}\label{S:staircase_def}

We use the definition of staircase diagram as in \cite{RS15}.
Let $\Gamma$ be a graph and $S$ be the set of vertices in $\Gamma$.
If $s,t \in S$ we write $s$ adj $t$ to mean that $s$ is adjacent to $t$.
Let $\mcD \subseteq 2^S$ and $\prec$, a poset on $\mcD$. Elements $\mcB \in \mcD$ are called \emph{blocks} of $\mcD$.
Let $\mcB \in \mcD$. We say $\mcB$ is connected if the induced subgraph of $\mcB$ is connected in $\Gamma$.
Recall that $\mcB$ covers $\mcB'$, in the partial order under $\prec$, if $\mcB \succ \mcB'$ and there is no $\mcB'' \in \mcD$ so that $\mcB \succ \mcB'' \succ \mcB'$.
A subset $C \subset \mcD$ is a chain if it is totally ordered.
A subset $C \subset \mcD$ is saturated if there is no $\mcB'' \in \mcD\setminus C$ with $\mcB \succ \mcB'' \succ \mcB'$ for any $\mcB, \mcB' \in \mcD$.

Given some vertex in $\Gamma$ we define $\mcD_s$ to be the blocks over $S$:
	\begin{equation*}
    		\mcD_s:=\{\mcB\in\mcD\ |\ s\in \mcB\}.
	\end{equation*}

We now define the staircase diagram:
\begin{definition}\label{D:staircase}

	Let $\mcD = (\mcD, \preceq)$ be a partially ordered subset of $2^S$ not
    containing the empty set. We say that $\mcD$ is a \emph{staircase diagram}
    if the following are true:
    \begin{enumerate}[(1)]
        \item Every $\mcB\in\mcD$ is connected, and if $\mcB$ covers $\mcB'$ then $\mcB\cup \mcB'$ is connected.
        \item The subset $\mcD_s$ is a chain for every $s \in S$.
        \item If $s\adj t$, then $\mcD_s\cup \mcD_t$ is a chain, and $\mcD_s$ and $\mcD_t$
            are saturated subchains of $\mcD_s \cup \mcD_t$.
        \item If $\mcB\in\mcD$, then there is some $s\in S$ (resp. $s'\in S$) such
            that $\mcB$ is the minimum element of $\mcD_s$ (resp. maximum element of
            $\mcD_{s'}$).
    \end{enumerate}
\end{definition}

In later parts of the paper we deal with type $E$ Dynkin diagrams and introduce other constructions on staircase diagrams over $E$-type diagrams. It is useful to have the following definition:

\begin{definition}\label{D:stargraph}
A star graph is defined to be a tree with the property that only $1$ node (called the branching vertex) has degree greater than $2$. A branch of a star graph is the path from a vertex of degree $1$ to the branching vertex and the length of a branch is the length of the corresponding path.
We let $\Gamma(a_0, a_1, \cdots, a_k)$ denote a graph with branches of length $a_0, \cdots, a_k$. We say the branch of length $a_0$ is the main branch of the graph. Let $v_b$ denote the branching vertex.
\end{definition}

For example, the following is a star graph $\Gamma(6,4,2)$.\\
\begin{example}\label{ex1}
\end{example}
\begin{tikzpicture}
[scale=.8,auto=left,every node/.style={circle,fill=blue!20}]
  \node (n6) at (10,0) {6};
  \node (n5) at (8,0)  {5};
\node (n4) at (6,0)  {4};
  \node (n3) at (4,0) {3};
  \node (n2) at (2,0) {2};
  \node	(n1) at (0,0)	{1};
  \node (n7) at (12,2)  {7};
  \node (n8) at (14,4)  {8};
  \node (n9) at (16,6)	{9};
  \node (n10) at (12,-2)  {10};

  \foreach \from/\to in {n1/n2, n2/n3, n3/n4, n4/n5, n5/n6, n6/n10, n6/n7, n7/n8, n8/n9}
    \draw (\from) -- (\to);

\end{tikzpicture}\\

And the following is a staircase diagram over that graph.\\

\begin{center}
\begin{tikzpicture}[scale=0.5]
            \planepartition{}{ {0,0,0,0,0,2,0,0,1,1},{0,0,1,0,1,1,1,1,0,1,1},{0,0,0,1,1,1,1}}.
        \end{tikzpicture}
\end{center}

\begin{defn} \cite{RS15}
    We define $\supp(\mcD)$ to be the set of vertices in the support of a staircase diagram.
    \begin{equation*}
        \supp(\mcD) := \bigcup_{\mcB \in \mcD} \mcB.
    \end{equation*}
    We say $\mcD$ is \emph{connected} if the support is a connected subset
    of the base graph. A subset $\mcD' \subset \mcD$ is a \emph{subdiagram} if
    $\mcD'$ is a saturated subset of $\mcD$.
\end{defn}

\begin{definition}
A staircase diagram, $D$, is fully supported (on the graph $\Gamma$) if for each $\gamma \in \Gamma$, $D_\gamma$ is non-empty.
\end{definition}

We now introduce tools to decompose staircase diagrams. Along each branch we will have $2$ parts, the broken tower and the regular part.
\begin{definition}\label{D:Diagram}
A broken tower over a vertex $s$ is the set $\mcD_{s}$. Unless otherwise stated, we consider broken tower diagrams over the branch vertex $v_b$.
We say that a restriction to a broken tower diagram over a vertex $s_1$ is a poset of $2^S$ containing only sets of the form $s_1, \cdots, s_l$ and satisfying axioms $(1) - (3)$ of the staircase diagram definition.
In the case that the poset also satisfies axiom $(1)-(4)$ then it is said to be a tower diagram.
\end{definition}

\begin{definition}\label{D:Rest}
Given a subgraph $H \subset G$ and a staircase diagram $\mcD$ over $G$ the restriction of $\mcD$ to $H$ is the set $\mcD|_H = \{\mcB \cap H | \mcB \in \mcD\}$.
 The restriction of a subgraph to a broken tower diagram of $\mcD$ over a branch $H \subset \Gamma$, $\Gamma$ being a star graph (with $v_b$ being the branch vertex) is the set $\mcD_{v_b}|_H = \{\mcB \cap H | \mcB \in \mcD_{v_b}\}$. Define $\mcB_l$ to be the maximum length block in the restriction to a broken tower, that is, so that $|\mcB_i|$ is maximal. In the case where there are $2$ or more blocks with maximum length we define $\mcB_l$ to be the set of all maximum length blocks.
\end{definition}
The following is an immediate consequence of the above definition.
\begin{lemma}
The $2$ definitions of a restriction to a broken tower are equivalent.
\end{lemma}
\begin{proof}
	 Over any path blocks containing $s_1$ and $s_m$ must contain $s_i$ for all $1 \leq i \leq m$ by axiom $(1)$.
\end{proof}

\begin{lemma}\label{l:rest}
For any staircase diagram $\mcD$, the restriction to an broken tower diagram is a broken tower diagram.
\end{lemma}
\begin{proof}
$\mcD_{v_b}|H$ satisfies axioms $(1) - (3)$ of the staircase diagram definition and each $\mcB \in \mcD_{v_b}$ has $v_b \in \mcB$.
\end{proof}
\begin{definition}\label{D:Reg}
Let $\mcD_{v_b}|_H$ be a restriction to a broken tower diagram over a branch $H$ of $\Gamma$. We define the regular part to be the set of blocks $Reg(\mcD) = \{\mcB \in \mcD | \mcB \notin \mcD_{v_b}|_H\} \cup (\mcB_l \setminus \mcB_{l+1})$ or $Reg(\mcD) = \{\mcB \in \mcD | \mcB \notin \mcD_{v_b}|_H\} \cup (\mcB_l \setminus \mcB_{l-1})$ (where $\mcB_l$ is the maximum block in the broken tower and $\mcB_{l \pm 1}$ is the block $1$ above or below $\mcB_l$ in the poset) with either $(\mcB_l \setminus \mcB_{l \pm 1}) \prec \mcB_1$ or  $(\mcB_l \setminus \mcB_{l \pm 1}) \succ \mcB_1$ (with $\mcB_1$ being the first block along the branch $\Gamma$).
\end{definition}
 \begin{lemma}
 The regular part of a staircase diagram is always a staircase diagram over Dynkin diagrams of type $A_n$.
 \end{lemma}
 \begin{proof}
 Since the regular part is along a path it must be of type $A_n$. Furthermore, all blocks of $Reg(\mcD)$ satisfies axioms $(1)-(4)$ of \ref{S:staircase_def} because they are a connected subset of a diagram $\mcD$.
 \end{proof}
 We can decompose the previous example to obtain the following broken tower diagram (which can be further decomposed into restrictions to broken towers along the $3$ branches) as well as $3$ regular parts (in this case, one regular part is empty).
\begin{example}
The broken tower diagram of \ref{ex1} is:\\
\begin{center}
	\begin{tikzpicture}[scale=0.5]
		 \planepartition{}{ {0,0,0,0,0,2,0,0,0},{0,0,0,0,1,1,1,1},{0,0,0,1,1,1,1}}.

	\end{tikzpicture}
\end{center}

With the $3$ regular parts being the empty-set and:
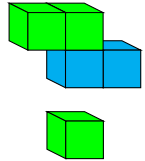
\begin{center}
\begin{tikzpicture}[scale=0.5]
            \planepartition{}{ {0,0,0,0,0,0,0,0,1,1},{0,0,0,0,0,0,0,0,0,1,1},{0,0,0,0,0,0,0}}.
        \end{tikzpicture}

\end{center}

\begin{center}
\begin{tikzpicture}[scale=0.5]
            \planepartition{}{ {0,0,0,0,0,0,0,0,0,0},{0,0,1,0,0,0,0,0,0,0,0},{0,0,0,0,0,0,0}}.
        \end{tikzpicture}
\end{center}

The restriction to a broken tower diagram on each branch would be the following $3$ diagrams. You'll notice we can piece together the entire staircase from this data.
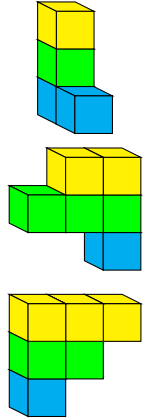
\begin{center}
	\begin{tikzpicture}[scale=0.5]
		 \planepartition{}{ {0,0,0,0,0,2,0,0,0},{0,0,0,0,0,1,0,0},{0,0,0,0,0,1,0}}.
		 	\end{tikzpicture}
\end{center}
\begin{center}
	\begin{tikzpicture}[scale=0.5]

\planepartition{}{ {0,0,0,0,0,1,0,0,0},{0,0,0,0,0,1,1,1},{0,0,0,0,0,1,1}}.

\end{tikzpicture}
\end{center}
\begin{center}
	\begin{tikzpicture}[scale=0.5]

		 \planepartition{}{ {0,0,0,0,0,1,0,0,0},{0,0,0,0,1,1,0,0},{0,0,0,1,1,1,0}}.
\end{tikzpicture}
\end{center}

\end{example}

 We now define the idea of a join and glue of $2$ diagrams. The idea being we can take the join of a regular part and a restriction to a broken tower to get back to a branch of the original staircase diagram after decomposing it into the $2$ parts. Similarly, gluing takes multiple branches of a diagram $\mcD$ and returns the staircase diagram with the branches ``glued'' along the branching vertex.

 \begin{definition}\label{D:J}
Given a restriction of broken tower diagram $\mcD_1$ (over some branch) and a regular part $\mcD_2$ over the same branch we define the join of the diagrams to be the diagram obtained by joining the largest block of $\mcD_1$ and the first block of $\mcD_2$, under the condition that $\mcB_l \setminus \mcB_{l \pm 1} \setminus S(\mcD_2\setminus \mcB_2) \ne \emptyset$, $\mcB_1 = \mcB_l \setminus \mcB_{l \pm 1}$ and $S(\mcD_1) \cap S(\mcD_2) \setminus \mcB_{1} = \emptyset$
(where $\mcB_1$ is the first block of $\mcD_2$ and $\mcB_2$ the second, along the branch $\Gamma$).\\

That is, $Join(\mcD_1,\mcD_2)$ is the poset $\mcD_1 \cup (\mcD_2 \setminus \mcB_1)$ with $\mcB_2 \prec \mcB_l$ or $\mcB_l \prec \mcB_2$. In the case that one of the conditions is not satisfied the join does not exist and it is said to be not valid.\\
\end{definition}

\begin{example}
	The join of \ref{ex1} along the main branch ($6$-nodes) is
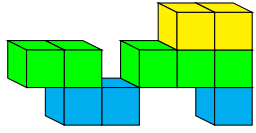
\begin{center}
\begin{tikzpicture}[scale=0.5]
            \planepartition{}{ {0,0,0,0,0,1,0,0,1,1},{0,0,0,0,0,1,1,1,0,1,1},{0,0,0,0,0,1,1}}.
        \end{tikzpicture}
\end{center}
	the ``union'' of the restricted tower diagram and the regular part.

\end{example}

\begin{definition}\label{D:G}
Given $2$ staircase diagrams (or broken tower/ restriction broken tower diagrams) containing $v_b$, $\mcD_1$ and $\mcD_2$, we denote $Glue(\mcD_1, \mcD_2)$ to be the diagram where the diagrams $\mcD_1$ and $\mcD_2$ are glued together at the vertex $v_b$. If $\mcB_1,\cdots, \mcB_n$ are blocks (in poset order) over the restriction to a broken tower diagram of $\mcD_1$ and $b_1, \cdots, b_m$, over $\mcD_2$, then
$$Glue(\mcD_1, \mcD_2) = \bigcup_{i \in \mathbb{Z}}(\mcB_i \cup b_{\rho(i)}) \cup Reg(\mcD_1) \cup Reg(\mcD_2).$$
With the order induced from $\mcD1$ and $\mcD2$,and $\rho(\alpha) = \alpha + x$, for some $x$ with $(-m \leq x \leq n)$ and $\mcB_i = \emptyset$ or $b_i = \emptyset$ for values of $i$ not between $1$ and $m$ or $n$, respectively. Given $n$ diagrams $\mcD_i$, $Glue_{i}(\mcD_i)$ is defined to be $Glue(\mcD_1, Glue (\mcD_2, \cdots Glue (\mcD_{n-1},\mcD_{n}) \cdots )$.
\end{definition}

For example, gluing the $3$ restricted tower diagrams above gives the broken tower diagram:
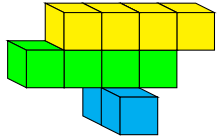
\begin{center}
	\begin{tikzpicture}[scale=0.5]
		 \planepartition{}{ {0,0,0,0,0,2,0,0,0},{0,0,0,0,1,1,1,1},{0,0,0,1,1,1,1}}.

	\end{tikzpicture}
\end{center}

\begin{lemma}
Note the following:
$$Glue(\mcD_1, Glue(\mcD_2,\mcD_3)) = Glue(Glue(\mcD_1,\mcD_2), \mcD_3).$$
Thus $Glue(\mcD_1, \cdots, \mcD_n)$ is well-defined.

\end{lemma}

\section{Decomposing Staircases}
We now prove that every staircase can be decomposed into a broken tower and regular part along each branch (with a choice of arrow for each branch). Furthermore, $n$ triples of regular branch, arrow and broken tower uniquely determine a star graph on with $n$ branches.
\begin{theorem}\label{bijection}
There is a bijection between staircase diagram over a star graph $G$ with $n$ branches and $n$ triples of regular part,
restriction to a broken tower diagrams, choice of $\uparrow, \downarrow$ of $G$ with the join of the regular parts and broken tower diagrams valid.
\end{theorem}
\begin{proof}
Any connected staircase diagram $\mcD$ over a star graph $G$ may be decomposed uniquely into $n$ triple of regular part, restriction to a broken tower diagram, and choice of arrow as follows:
Let the branches be $A_i$ with branch vertex $v_b$. For each branch ($1 \leq i \leq n$)we use the following triple:
The broken tower diagram is $\mcD_{v_b}|_{A_i}$.
The regular part is $Reg(\mcD|_{A_i})$.
The choice of arrow is defined by whether $(\mcB_l \setminus \mcB_{l \pm 1}) \prec \mcB_1$ or  $(\mcB_l \setminus \mcB_{l \pm 1}) \succ \mcB_1$.
Since each of these are defined uniquely this is indeed unique and every connected staircase diagram can be decomposed this way.\\

For the other direction assume we are given $n$ triples of regular part $S_i$, restriction to broken tower diagrams $T_i$ and an arrow. Then
$$Glue_{i = 1}^{n}(Join(S_i,T_i))$$
is a unique staircase diagram. Here each $\uparrow\downarrow$ denotes whether  $\mcB_2 \prec \mcB_l$ or $\mcB_l \prec \mcB_2$ in the join.
\end{proof}
The advantage of this is that we can focus on enumerating broken parts and regular parts almost individually. The following lemma determines the form of all broken tower diagrams. The blocks must form and increasing then decreasing sequence in length. An example of a potential broken tower diagram is
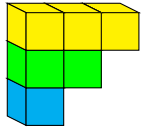
\begin{center}
	\begin{tikzpicture}[scale=0.5]
		 \planepartition{}{ {0,0,0,0,0,1,0,0,0},{0,0,0,0,1,1,0,0},{0,0,0,1,1,1,0}}.

	\end{tikzpicture}
\end{center}

while a non-example (as disproved in \ref{tri}) is
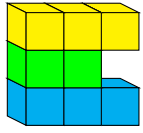
\begin{center}
	\begin{tikzpicture}[scale=0.5]
		 \planepartition{}{ {0,0,1,1,1,0,0,0,0},{0,0,0,1,1,0,0,0},{0,0,1,1,1,0,0}}.

	\end{tikzpicture}
\end{center}

\begin{lemma}\label{tri}
Let $\mcD$ be a restriction of a broken tower diagram on the graph $\Gamma$ with vertices $s_1,s_2, \cdots s_p$  (the branch vertex is $v_b$) blocks $\mcB_1 \prec \mcB_2 \cdots \prec \mcB_n \in \mcD$ each with cardinality $|\mcB_i| = l_i$. Then there exist $i,j \in \mathbb{N}$ such that $l_1 \leq l_2 \leq \cdots < l_i = l_{i+1}, \ldots, l_{i+j} > \ldots \geq l_{n-1} \geq l_n$.
\end{lemma}
\begin{proof}
	We proceed by induction on $n$. For $n = 1$ the statement holds. Now suppose we have a restriction of broken tower staircase diagram of with $n$ blocks ordered by $\mcB_1 \prec \mcB_2 \cdots \prec \mcB_n \in \mcD$. And suppose we insert another block, $\mcB_q$ with $|\mcB_q| = l_q$ into the poset between $2$ blocks $\mcB_{k-1}$ and $\mcB_{k}$. Recall that $i$ and $j$ give the boundaries for the maximum sized blocks from the inductive assumption. We have $3$ cases, $k < i$, $i \leq k \leq i+j+1$ or $k > i+j+1$ and we will now show that in each case we can find new $i,j$ to satisfy the theorem if $\mcD$ is still a restriction of broken tower staircase diagram.

{\bf Case 1:} If $k<i$ then $l_{k-1} \leq l_{q} \leq l_{k}$, otherwise we have $l_{k-1} > l_q$ or $l_{n-1} > l_{k}$ the former of which implies $\mcD_{s_{l_q+1}}$ is not a saturated subchain of $\mcD_{s_{l_q}} \cup \mcD_{s_{l_q+1}}$ (since $\mcD_{s_{l_q}} \notin \mcD_{s_{l_q+1}}$ meaning that $\mcB_{s_k} \prec \mcB_{s_{l_q}} \prec \mcB_{s_{k+1}}$)
, breaking axiom $(3)$ of the broken tower staircase diagram definition. The latter implies $\mcD_{s_{l_{k}+1}}$ is not a saturated subchain of $\mcD_{s_{l_{k}+1}} \cup \mcD_{s_{l_{k}}}$ (since $\mcD_{s_{l_k}} \notin \mcD_{s_{l_k+1}}$ but $\mcD_{s_{l_q}} \in \mcD_{s_{l_k+1}}$ meaning that  $\mcB_{s_k} \prec \mcB_{s_{l_q}} \prec \mcB_{s_{k+1}}$), again breaking axiom $(3)$ of the broken tower staircase diagram definition.

{\bf Case 2:} If $i \leq k \leq i+j+1$ then $l_q \geq \min(l_{k},l_{k-1})$. Otherwise we have $l_q < l_{k}$ and $l_q < l_{k+1}$. Then $\mcD_{s_{l_q+1}}$ is not a saturated subchain of $\mcD_{s_{l_q+1}} \cup \mcD_{s_{l_q}}$ since $\mcB_q \in \mcD_{s_{l_{q}}}$ and $\mcB_{k-1} \prec \mcB_q \prec \mcB_k$ but $\mcB_q \notin \mcD_{s_{l_q+1}}$. Again, breaking axiom $(3)$ or the broken tower staircase diagram definition.

{\bf Case 3:} If $k > i+j+1$ then we may repeat Case $1$, mutatis mutandis, by symmetry.
\end{proof}
\begin{definition}
For each block in the restriction to a broken tower diagram we assign an element of the set $\{\emptyset, u, d, u/d \}$ depending on whether the block has a maximal element, minimal element, both or neither ($u, d, u/d$ or $\emptyset$ respectively ). We call the ordered set of such elements the sight-set of the particular broken tower (over the branching node of that tower). If a sight-set consists of only $u/d$ elements then we say it is full and corresponds to a valid tower over the branch node (a tower that is also a staircase diagram of definition \ref{D:staircase}).
\end{definition}

\begin{corollary}\label{l:cor}
The sight-sets over the branch vertex of a restriction to a broken tower diagram is of form $u\{u,\emptyset\}^*\{u/d\}\{d,\emptyset\}^*d$, $u\{u,\emptyset\}^*\{d,\emptyset\}^*d$, ${\{u/d\}\{d,\emptyset\}^*d}$, $u\{u,\emptyset\}^*\{u/d\}$ or $\{u/d\}$. The notation used here is formal language notation for strings containing the $4$ symbols $\{u, d, u/d, \emptyset \}$. We call any particular form the gluing data of broken tower diagram.
\end{corollary}
\begin{proof}
From lemma \ref{tri} we have that the length or cardinality of the blocks form an increasing then decreasing sequence. For any block $\mcB_{k}$ if $|\mcB_{k-1}| < |\mcB_k|$ then the $k^{th}$ term of the sight-set has a $u$ since $\mcB_k$ has a maximal element. Similarly, if $|\mcB_{k}| > |\mcB_{k+1}|$ then the $k^{th}$ term of the sight-set has a $d$ since $\mcB_k$ has a minimal element. The above forms are a consequence of the increasing then decreasing sequence with $\emptyset$ occurring when $|\mcB_{k-1}| = |\mcB_k| = |\mcB_{k+1}|$
\end{proof}

\begin{theorem}\label{adding lemma}
Blocks added onto a broken tower diagram (via a join to a regular part) don't change the sight set on the branching vertex.
\end{theorem}

\begin{proof}
Firstly, we cannot increase the sight set since none of the added blocks are blocks on the branching vertex. Suppose (W.L.O.G) that adding a block onto a broken tower diagram caused a $u$ to be removed from the sight set. The $u$ removed cannot be from $\mcB_1$ since that would imply the branching vertex was covered. Suppose $u$ was removed from $\mcB_i$, then this would imply that the new block (say, $\mcB_\nu$) is adjacent to $\mcB_{i-1}$ in poset order (otherwise there is some chain with an element in $\mcB_\nu$ that is not saturated ). From the valid join condition we have that $\mcB_{i} \nsubseteq \mcB_\nu \cup \mcB_{i-1}$ thus $\mcB_{i}$ still has a maximal block and the $u$ is preserved. Contradiction, the join cannot be valid and sight-set reducing.
\end{proof}

\begin{theorem}\label{height}
$|\mcD_{v_b}|$ is bounded as $n$ varies. Furthermore, $|\mcD_{v_b}| \leq \sum_{i}a_{i} + 1$.
\end{theorem}
\begin{proof}
Let $\mcD_{v_{b}} = \{\mcB_{1},\cdots,\mcB_{h}\}$ with $\mcB_{i} \succ \mcB_{j}$ if $i < j$. Since $\mcD$ is a staircase diagram each block is minimal over some vertex and maximal over some vertex. Consider $\mcD$ restricted to the subgraph consisting of the segment from the vertex $v_b$ to the end of the branch of size $n$. Consider the sight-set of  each restriction to broken tower.  By lemma
\ref{adding lemma} we can further restrict ourselves to the case where $\mcD$  is now a restricted broken tower diagram since adding blocks to broken tower diagrams does not change the sight-set over the branching vertex. The form of the sight set will be one of the forms as described in \ref{l:cor}
(where the size of the ordered set is $h$). Since $\mcD$ is a restriction to a broken tower diagram, we must have $h-1$ symbols (that is, either $u$ or $d$) added when we glue (at $v_b$) the broken tower diagram of $\mcD$ restricted to the other branches of $\Gamma$.\\

 For each branch of $\Gamma$ we can restrict to the broken tower diagram by lemma \ref{adding lemma} and for each vertex $v$ in the other branches we can have $2$ symbols from the minimal and maximal elements in $\mcD_v$. However, from the $2$ symbols coming from $\mcD_v$ only $1$ such symbol will be able to change the sight-set of $\mcD$ restricted, as if we had both begin used this would imply that the minimal block is some $\mcB_{i}$, the maximal block is some $\mcB_{j}$ with $\mcB_{i} \succ \mcB_{j}$ since all the blocks containing $u$ are greater in the poset than blocks containing $d$ (in $\mcD$ restricted). This implies to add $h-1$ symbols we need $\sum_{i}a_i \geq h-1$ or $$h \leq \sum_{i}a_i + 1.$$
\end{proof}

\section{The fully supported $E$-type case}
We now move on to prove the main results of the paper, the complete enumeration of rationally smooth Schubert varieties over type $E$ Dynkin diagrams.

\begin{definition}
We let $E_n$ denote the Dynkin diagram of the corresponding Coxeter group (the finite type Dynkin diagram of type $E$ with $n$ nodes). Let $E(x) = \sum_{i}e_n x^n$ where $e_n$ is the number of fully supported connected staircase diagrams over the graph $E_n$.
\end{definition}
\begin{lemma}\label{sight-set fix}
Suppose we have $2$ restriction to a broken tower diagram $\mcD_1$ and $\mcD_2$ with the same sight-set over a graph with $n$ branches. Then the set of all $n-1$ tuples of pairs of restricted broken towers and regular parts so that the glue and join of them all with the first restriction to a broken tower diagram, $\mcD_1$ (so that a staircase diagram is formed, as described by theorem \ref{bijection}) is the same as that for the second, $\mcD_2$.
\end{lemma}
\begin{proof}
Suppose we have $2$ different restrictions to a broken tower with the same sight-set but a  $n-1$ tuple of pairs of diagrams that gives a full sight-set over the branching node of only $\mcD_1$. Then there is some $\mcB \in \mcD_2$ so that when glued with the pairs does not give a $u/d$. Since $\mcD_1$ has the same sight-set as $\mcD_2$ that means that there is some broken tower in the set of $n-1$ with a different sight-set. Contradiction.
\end{proof}
\begin{definition}
Fix a gluing data $t$ across $1$ branch of the star graph $\Gamma$ with $n$ branches. For any gluing data $t$, we define $Cont(t)$ to be the number of $n-1$ tuples consisting of pairs of restrictions of broken towers along the other branches and regular parts so that the join across each of the pairs and glue between them all has a full sight-set over the branch vertex.
\end{definition}
\begin{corollary}
As a direct corollary of \ref{height},
we have that $|\mcD_{v_b}| \leq 4$. From now on we let $V_1, V_2, V_3, V_4$ denote these blocks from greatest to least in poset order, in the case that less than $4$ are needed, define the extra blocks to be empty sets.
\end{corollary}

\begin{lemma}
$$E(x) = x^3 \sum_{t \in \mathcal{T}} Cont(t)\sum_{BT_t}SD_{BT_t}(x)$$
Here $\mathcal{T}$ is the set of all possible gluing data, $BT_t$ is any broken tower diagram of gluing data $t$ and $SD_{BT_t}(x)$ is the generating function for joins of regular part and broken tower diagrams of gluing data $t$ (with exponent indicating length of the main branch).
\end{lemma}
\begin{proof}
We have that any staircase diagram on $E$ can be decomposed into a triple of pairs of regular part, broken tower along each branch. Fix some gluing data $t$ for the broken tower along the main branch, by lemma \ref{sight-set fix}
we have that the number of pairs of regular part, broken tower for the other $2$ branches is fixed, thus we only concern ourselves with pairs of regular part and broken tower diagrams. Given a generating function for the number of such pairs of length $n-3$ we simply introduce a $x^3$ term for the $3$ additional vertices past the branch. Summing $t$ over all gluing data gives the required.
\end{proof}

\begin{lemma}\label{special cat}
The number of dyck paths from $(k,1)$ to $(n,n)$ is given by $$\frac{\binom{2n-k}{n-k}k}{2n-k}.$$
This is a consequence of the ballot theorem, for a proof see p. 73 of \cite{feller1}.
\end{lemma}

\begin{definition}
A broken staircase diagram is partially ordered set collection of subsets $\mathcal{B}$ (over $A_n$) so that $\mathcal{B} = \{\mcB \cap S_n |\mcB \in \mcD\}$ where $\mcD$ is a increasing or decreasing staircase diagram.
\end{definition}
\begin{lemma}\label{stair cat}
There is a bijection between Dyck paths from $(k,1)$ to $(n,n)$ and increasing connected staircase diagrams starting with a block of length $k$. An increasing staircase diagram has the property that $\mcB_1 \prec \mcB_2 \prec \cdots \mcB_n$ if $\mcB_i$ is the $i^{th}$ block along the branch. This is the same bijection as in \cite{RS15}.
\end{lemma}
\begin{proof}
Let $r_i$ be a sequence of integers that form a dyck path with $i^{th}$ step being of length $r_i$ with the property that $\sum_{i} r_{2i} = n$ and $\sum_{i}r_{2i-1} = n$ (since this is a path to $(n,n)$) and $r_1 = k$. Define the blocks $\mcB_i$ over the vertices $s_1, s_2, \cdots$ to be $$\mcB_i = \{s_{\sum_{k=1}^{i-1}r_{2k} +1},\cdots , s_{\sum_{k=1}^{i}r_{2k-1} + 1}\}.$$The poset order on these blocks is $\mcB_1\prec \mcB_2 \cdots$. This is indeed a staircase diagram as one can easily see conditions $(1) - (3)$ from the staircase diagram definition are satisfied.
For the final condition we see that both ${\sum_{k=1}^{i-1}r_{2k} +1}$ and ${\sum_{k=1}^{i}r_{2k-1} + 1}$ are increasing with $i$ so all blocks have minimal and maximal elements. Also, note that the diagonal condition on Dyck paths makes sure that $\mcB_i$ is always defined, since this forces ${\sum_{k=1}^{i-1}r_{2k} +1} \leq {\sum_{k=1}^{i}r_{2k-1} + 1}$. This map from Dyck paths to staircase diagrams is injective (suppose that $2$ different Dyck paths differ at the $2i^{th}$ or $2i-1^{th}$ step, then the corresponding block $\mcB_i$ is different).

 Finally, every staircase diagram starting with a block of length $k$ can be represented as a Dyck path  (with $r_1 = k$) by simply taking the difference of successive vertices in $\mcB_i$ to get a sequence of $r_i$. The Dyck path constraint that $\sum_{i}r_{2i-1} \geq \sum_{i}r_{2i}$ is satisfied because in any increasing staircase diagram $\mcD$ we can consider the subdiagram of the first $i$ blocks. The subdiagram has $\sum_{i}r_{2i-1}$ minimal elements and at least $ \sum_{i}r_{2i}$ maximal elements (exactly  $\sum_{i}r_{2i}$ if we only consider maximal elements that are maximal within $\mcD$ as well). Since the number of maximal and minimal elements on a diagram must be equal ($1$ of each for every vertex) we have that $$\sum_{i}r_{2i-1} \geq \sum_{i}r_{2i}.$$
\end{proof}
\begin{lemma}\label{IK}
Let $I_k(x)$ be the generating function for increasing connected staircase diagrams starting with a block of length $k$. Let $IB_k(x)$ be the generating function for increasing broken staircase diagrams starting with a block of length $k$. Then:
\begin{enumerate}
\item $$I_k(x) = \sum_{n}\frac{\binom{2n-k}{n-k}k}{2n-k}x^n$$
\item $$IB_k(x) = \frac{(1-x)I_{k}(x)}{x} - x^{k-1} + x^k$$
\end{enumerate}
\end{lemma}
\begin{proof}
The generating function for $I_k(x)$ come directly from \ref{special cat} and \ref{stair cat}. To find the generating function for $IB_k(x)$ note that the restriction to $\{s_1, \cdots, s_{n-1}\}$ of any increasing staircase diagram starting with a block of length $k$  on $\{s_1, \cdots,s_{n}\}$ is either an increasing broken staircase diagram or another increasing staircase diagram starting with a block of length $k$ except for the special case where the staircase diagram is just a block of length $k$. This information gives the following:
$$x^k + xIB_{k}(x)+xI_{k}(x) = I_{k}(x) + x^{k+1}$$
Yielding the required.
\end{proof}
\begin{lemma}
The generating function for $B_k(x)$, that is, staircase diagrams starting with a first block of length $k$ and second block greater than the first in poset order ($\mcB_1 \prec \mcB_2$) on the graph $A_n$ is given by
$$B_k(x) = \frac{I(x)IB_k(x)}{1-Br(x)^2}.$$
Here $I(x)$ and $Br(x)$ are the generating functions for increasing staircase diagrams and increasing broken staircase diagrams as found in \cite{RS15}.
$$I(x) = \frac{1 - 2x \sqrt{1-4x}}{2x}$$
$$Br(x) = \frac{(1-x)I(x)}{x} - 1 $$
\end{lemma}
\begin{proof}
Let $\mcD$ be a broken staircase diagram of type $\mcB_k$. We define a critical point of $\mcD$ as an node $s \in \Gamma$ in the underlying graph so that $s$ is both maximal and minimal in the chain, or equivalently, that $|\mcD_s| = 1$.\\

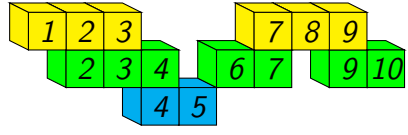
\begin{example}
   \begin{tikzpicture}[scale=0.5]
\planepartitionD{10}{{0,0,0,0,0,1,1,0,0,0},{1,1,0,1,1,0,1,1,1,0},{0,1,1,1,0,0,0,1,1,1}}
   \end{tikzpicture}

Has critical points at $s_1, s_5, s_8$ and $s_{10}$.
\end{example}

Suppose we break diagram $\mcD$ into parts $\mcD_1, \mcD_2, \cdots \mcD_{2n}$ so that $\bigcup_{i = 1}^{2n}\mcD_i = \mcD$ where $\mcD_1$ is an increasing diagram starting with a first block of length $k$, $\mcD_{2n}$ is a decreasing broken staircase diagram (including the empty staircase diagram) and $\mcD_i$ for $2 \leq i \leq 2n-1$ is a broken staircase diagram.We may represent the generating function for such diagrams as $I(x)IB_{k}(x)Br(x)^{2n-2}$. Summing over $n$ gives:
$$\sum_{n = 1}^{\infty} I(x)IB_{k}(x)Br(x)^{2n-2} = \frac{I(x)IB_k(x)}{1-Br(x)^2}.$$
\end{proof}

\begin{definition}
We use $SD_{\mcD}(x)$ to denote the generating function (with respect to length of the branch) enumerating the number of joins of staircase diagrams with restriction to a broken tower, $\mcD$.
\end{definition}
\begin{lemma}\label{SD}
Suppose $\mcD_1$ is an broken tower diagram with maximum block length $|\mcB_l| = k$ and suppose that $|\mcB_{<l}| = x_1$ and $|\mcB_{>l}| = x_2$. Here $\mcB_{<l}$ denote the block $\mcB_{i} \succ \mcB_{l}$ and $|\mcB_i| < k$. Similarly, $\mcB_{>l}$ denote the block $\mcB_{i} \prec \mcB_{l}$ and $|\mcB_i| < k$. Then $$SD_{\mcD_1}(x) = x^k + x^{x_1} B_{k-x_1}(x) + x^{x_2}B_{k-x_2}.$$
\end{lemma}
\begin{proof}
Suppose that we have some join of a regular part and a restriction to a broken tower. As before we let $\mcB_1$ and $\mcB_2$ be the first $2$ blocks of the regular part. We have two cases: $\mcB_l \prec \mcB_2$ or $\mcB_l \succ \mcB_2$, each case corresponds to a term (either $x^{x_1} B_{k-x_1}(x)$ or $x^{x_2} B_{k-x_2}(x)$).
W.L.O.G. assume  $\mcB_1 \prec \mcB_2$. Assume we have a join of a regular part and a restriction to a broken tower, $\mcD$, over a path of $n$ vertices. From the definition of join we have that $\mcB_1 = \mcB_l \setminus \mcB_{l - 1}$ giving $|\mcB_1| = k - x_1$. Apart from this there are no more conditions on the regular part to be a valid join to $\mcD$. Since the first $x_1$ vertices of the path  on $n$ are taken up solely by the restriction to a broken tower the number of valid regular parts is the $x^{n-x_1}$ coefficient of $B_{k-x_1}(x)$. Summing over $n$ gives the desired generating function (and mutatis mutandis for $\mcB_l \succ \mcB_2$). We also add an $x^k$ term for no join:
$$SD_{D, \uparrow}(x) = \sum_{n}[x^{n-x_1}]B_{k-x_1}(x) = \sum_{n}[x^n]x^{x_1}B_{k-x_1}(x) = x^{x_1}B_{k-x_1}(x)$$
$$SD_{D, \downarrow}(x) = \sum_{n}[x^{n-x_2}]B_{k-x_2}(x) = \sum_{n}[x^n]x^{x_2}B_{k-x_2}(x) = x^{x_2}B_{k-x_2}(x)$$
$$SD_{D}(x) = SD_{D, \uparrow}(x) + SD_{D, \downarrow}(x) + x^k = x^k + x^{x_1} B_{k-x_1}(x) + x^{x_2}B_{k-x_2}$$
\end{proof}
\begin{lemma}\label{BT}
$$\sum_{BT_t} SD_{BT_t}(x) = \sum_{k}\sum_{(V_1, V_2, V_3, V_4) \in BT_{t,k}}x^{x_1}B_{k-x_1}(x) + x^{x_2}B_{k-x_2}.$$
The quadruple $(V_1, V_2, V_3, V_4)$ is in $BT_{t,k}$ if the broken tower diagram consisting of $V_1 \succ V_2 \succ V_3 \succ V_4$ has $|B_l| = k$ and has gluing data $t$.
\end{lemma}
\begin{proof}
This comes directly from \ref{SD} and the fact that $BT_t = \bigcup_{k = 1}^{\infty}BT_{t,k}$.
\end{proof}
\begin{lemma}\label{cases}
What follows is a table of values detailing the set of possible gluing data in the $E_n$ case as well parametrizations for the quadruples of $V_i$ in each $BT_{t,k}$ and the number of pairs of pairs of restriction to broken tower and regular part so that the glue of all the broken towers has a full sight-set. This is easily brute forceable in the $E_n$ case as the number of nodes outside the main branch is fixed. Given a gluing data one can test all the possible continuations.

\begin{table}[]

\begin{tabular}{|l|l|l|l|l|}
\hline
Broken Tower & Sight-Set                                                         & Parametrization & Continuations & Symmetry ($\times$2)                                                       \\ \hline
            \begin{tikzpicture}[scale=0.5]
            \planepartition{}{ {1},{1,1},{1,1,1}, {1,1,1,1}}.
        \end{tikzpicture}	& \begin{tabular}[c]{@{}l@{}}$u/d$\\ $d$\\ $d$\\ $d$\end{tabular}& $k > x_2 > |V_3| > |V_4| > 0$& 8             & Yes                                                                   \\ \hline
            \begin{tikzpicture}[scale=0.5]
            \planepartition{}{ {1},{1,1},{1,1}, {1,1,1}}.
        \end{tikzpicture}	& \begin{tabular}[c]{@{}l@{}}$u/d$\\ $\emptyset$\\ $d$\\ $d$\end{tabular} & $k > x_2 > |V_4| > 0$   &1         & Yes                                                                   \\ \hline
            \begin{tikzpicture}[scale=0.5]
            \planepartition{}{ {1},{1},{1,1}, {1,1,1}}.
        \end{tikzpicture}	& \begin{tabular}[c]{@{}l@{}}$u/d$\\ $d$\\ $\emptyset$\\ $d$\end{tabular} & $k> x_2 > |V_3| > 0$    &4          & Yes                                                                   \\ \hline
            \begin{tikzpicture}[scale=0.5]
            \planepartition{}{ {1},{1,1},{1,1,1,1}, {1,1,1}}.
        \end{tikzpicture}	& \begin{tabular}[c]{@{}l@{}}$u$\\ $u/d$\\ $d$\\ $d$\end{tabular}           & $k > x_1 , k > x_2 > |V_4|> 0$ &1              & Yes                                                           \\ \hline
            \begin{tikzpicture}[scale=0.5]
            \planepartition{}{ {1},{1,1}, {1,1,1}}.
        \end{tikzpicture}	& \begin{tabular}[c]{@{}l@{}}$u/d$\\ $d$\\ $d$\end{tabular}               &     $k > x_1 > |V_3| > 0$             &29               & Yes                                                                   \\ \hline
            \begin{tikzpicture}[scale=0.5]
            \planepartition{}{ {1},{1}, {1,1}}.
        \end{tikzpicture}& \begin{tabular}[c]{@{}l@{}}$u/d$\\ $\emptyset$\\ $d$\end{tabular}     &   $k > x_1 > 0$                         &13               & Yes                                                                   \\ \hline
            \begin{tikzpicture}[scale=0.5]
            \planepartition{}{ {1},{1,1}, {1,1}}.
\end{tikzpicture}	& \begin{tabular}[c]{@{}l@{}}$u$\\ $d$\\ $d$\end{tabular}                 &     $k > x_1 > 0$                         &1            & Yes                                                                   \\ \hline
\end{tabular}
\end{table}

\begin{table}[]

\begin{tabular}{|l|l|l|l|l|}
\hline
Broken Tower & Sight-Set                                                         & Parametrization & Continuations & Symmetry (x2)                                                       \\ \hline
            \begin{tikzpicture}[scale=0.5]
            \planepartition{}{ {1},{1,1}, {1}}.
\end{tikzpicture}	& \begin{tabular}[c]{@{}l@{}}$u$\\ $u/d$\\ $d$\end{tabular}               &  $k > x_1 > 0, k > x_2 > 0$         &12             & \begin{tabular}[c]{@{}l@{}}Not if $x_1 = x_2$\end{tabular} \\ \hline
            \begin{tikzpicture}[scale=0.5]
            \planepartition{}{ {1},{1,1}}.
\end{tikzpicture}	& \begin{tabular}[c]{@{}l@{}}$u/d$\\ $d$\end{tabular}                   &     $k > x_2 > 0$                            &41              & Yes                                                                   \\ \hline
            \begin{tikzpicture}[scale=0.5]
            \planepartition{}{ {1},{1}}.
\end{tikzpicture}		& \begin{tabular}[c]{@{}l@{}}$u$\\ $d$\end{tabular}                     &   $k$                                             &33            & No                                                                    \\ \hline
            \begin{tikzpicture}[scale=0.5]
            \planepartition{}{{1}}.
\end{tikzpicture}	& $u/d $                                                                                    & $k$                                              &10                              & No                                                                    \\ \hline
\end{tabular}
\end{table}
\newpage
\begin{table}[]

\begin{tabular}{|l|l|}
\hline
Broken Tower & Generating Function  \\ \hline
\begin{tikzpicture}[scale=0.5]
            \planepartition{}{ {1},{1,1},{1,1,1}, {1,1,1,1}}.
        \end{tikzpicture}					& $\sum _{k=1}^{\infty } \left(\sum _{x_1=1}^{k-1} \frac{(x_1-2) (x_1-1)}{2}  x^{x_1} B_{k-x_1}(x)\right)+\frac{(k-3) (k-2) (k-1)}{6}B_{k}(x)+x^k$  \\ \hline
\begin{tikzpicture}[scale=0.5]
            \planepartition{}{ {1},{1,1},{1,1}, {1,1,1}}.
    \end{tikzpicture}					& $\sum _{k=1}^{\infty } \left(\sum _{x_1=1}^{k-1} (x_1-1) x^{x_1} B_{k-x_1}(x)\right)+\frac{1}{2} (k-2) (k-1)B_{k}(x)+x^k$   \\ \hline
\begin{tikzpicture}[scale=0.5]
            \planepartition{}{ {1},{1},{1,1}, {1,1,1}}.
        \end{tikzpicture}					& $\sum _{k=1}^{\infty } \left(\sum _{x_1=1}^{k-1} (x_1-1) x^{x_1} B_{k-x_1}(x)\right)+\frac{1}{2} (k-2) (k-1)B_k(x)+x^k$ \\ \hline
\begin{tikzpicture}[scale=0.5]
            \planepartition{}{ {1},{1,1},{1,1,1,1}, {1,1,1}}.
        \end{tikzpicture}					& $\sum _{k=1}^{\infty } \left(\sum _{x_1=1}^{k-1} \frac{(k-2) (k-1) }{2} x^{x_1} B_{k-x_1}(x)+(k-1) (x_1-1) x^{x_1} B_{k-x_1}(x)\right)+\frac{(k-2) (k-1)^2}{2}  x^k$  \\ \hline
\begin{tikzpicture}[scale=0.5]
            \planepartition{}{ {1},{1,1}, {1,1,1}}.
    \end{tikzpicture}					&  $\sum _{k=1}^{\infty } \left(\sum _{x_1=1}^{k-1} (x_1-1) x^{x_1} B_{k-x_1}(x)\right)+\frac{1}{2} (k-2) (k-1)B_{k}(x)+x^k$  \\ \hline
\begin{tikzpicture}[scale=0.5]
            \planepartition{}{ {1},{1}, {1,1}}.
        \end{tikzpicture}					& $\sum _{k=1}^{\infty } \left(\sum _{x_1=1}^{k-1} x^{x_1} B_{k-x_1}(x)\right)+(k-1) B_{k}(x)+(k-1) x^k$   \\ \hline
\begin{tikzpicture}[scale=0.5]
            \planepartition{}{ {1},{1,1}, {1,1}}.
\end{tikzpicture}							& $\sum _{k=1}^{\infty } \left(\sum _{x_1=1}^{k-1} x^{x_1} B_{k-x_1}(x)\right)+(k-1) B_{k}(x)+(k-1) x^k$ \\ \hline
\begin{tikzpicture}[scale=0.5]
            \planepartition{}{ {1},{1,1}, {1}}.
\end{tikzpicture}							& $\sum _{k=1}^{\infty } 2(k-1) \left(\sum _{x_1=1}^{k-1} x^{x_1} B_{k-x_1}(x)\right)+(k-1)^2 x^k$     \\ \hline
\begin{tikzpicture}[scale=0.5]
            \planepartition{}{ {1},{1,1}}.
\end{tikzpicture}							& $\sum _{k=1}^{\infty } \left(\sum _{x_1=1}^{k-1} x^{x_1} B_{k-x_1}(x)\right)+(k-1) B_{k}(x)+(k-1) x^k$  \\ \hline
\begin{tikzpicture}[scale=0.5]
            \planepartition{}{ {1},{1}}.
\end{tikzpicture}							& $\sum _{k=1}^{\infty } 2B_k(x)+x^k$  \\ \hline
\begin{tikzpicture}[scale=0.5]
            \planepartition{}{{1}}.
\end{tikzpicture}							& $\sum _{k=1}^{\infty } 2B_k(x)+x^k$   \\ \hline

\end{tabular}
\end{table}
\end{lemma}

\begin{theorem}\label{main}
The generating function for $$E(x) = \sum_{i}e_n x^n = x^3\sum_{t \in \mathcal{T}} Cont(t) \sum_{k}\sum_{(V_1, V_2, V_3, V_4) \in BT_{t,k}}x^{x_1}B_{k-x_1}(x) + x^{x_2}B_{k-x_2}$$ is
$$\frac{P(x)+Q(x)\sqrt{1-4x}}{R(x)}.$$
Where \tiny
$$P(x) =  -3224 x^{16}+6436 x^{15}+2494 x^{14}-12226 x^{13}+7795 x^{12}-6091 x^{11}+14594 x^{10}-18859 x^9+$$
$$13325 x^8-4899 x^7+569 x^6+350 x^5-363 x^4+217 x^3-70 x^2+8 x$$
$$\hspace{-19mm}Q(x) = -2034 x^{15}+7488 x^{14}-9888 x^{13}+5071 x^{12}+815 x^{11}-5070 x^{10}+8318 x^9-$$
$$7544 x^8+3537 x^7-681 x^6-79 x^5+146 x^4-125 x^3+54 x^2-8 x$$
\normalsize
and
$$R(x) = (x-1)^4 \left(x^2+4 x-1\right).$$
\end{theorem}
\begin{proof}
This is obtained from the expression in \ref{BT} with the data obtained from \ref{cases} in $3$ parts. We sum those parts containing pure $x^k$ terms, those containing terms with $x_1 = 0$ ($B_k(x)$ terms) and those containing $B_{k-x_1}(x)$ terms. For each broken tower this gives a distinct generating function, adding each case together gives the theorem.
\end{proof}

\section{$E$-type Diagrams without full support}
\begin{theorem}
	The generating function $E_{Disc.}(x)$, enumerating the total number of staircase diagrams over the graph $E_n$ (including disconnected) is the rational $E_{Disc.}$, given below.
\end{theorem}
\begin{proof}
From \cite{RS15} we are given the generating function for diagrams and fully supported diagrams of type $A_n$ and $\mcD_n$. We denote these by $A_{Disc.}(x)$, $A(x)$, $\mcD_{Disc.}(x)$, and $D(x)$, respectively.To enumerate the total number of staircase diagrams over the graph $E_n$ we consider $4$ key vertices and apply an inclusion-exclusion argument. The vertices in question are vertices in the $2$ minor branches and the branch vertex. Any vertex can either have blocks lying over it or not. In the following diagrams an $X$ denotes no blocks lying over this vertex while an $O$ denotes blocks do lie over this vertex. The generating function for each possibility is also given. Note that in the case where he branch vertex is non-empty ($|D_{v_b} \ne \emptyset|$) we must multiply by a factor of $1 + A_{Disc.}(x)$ because the diagram is either fully supported on the main branch or is not and has another diagram (of type $A_{Disc.}(x)$ within it).

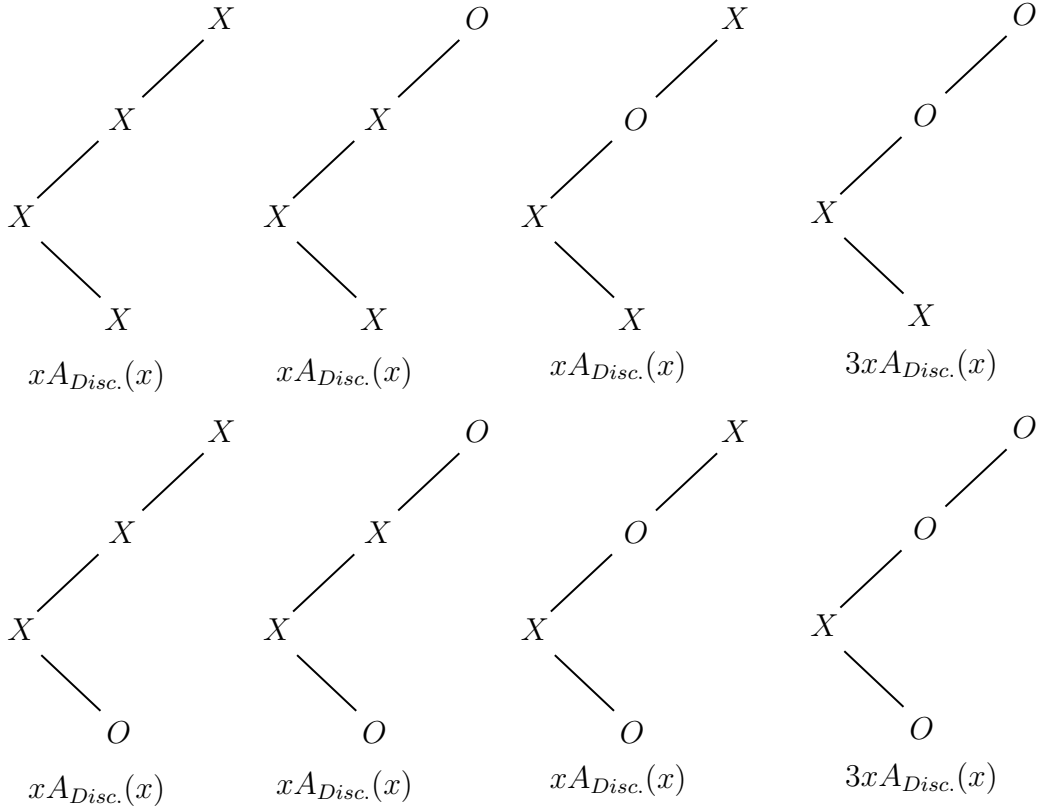
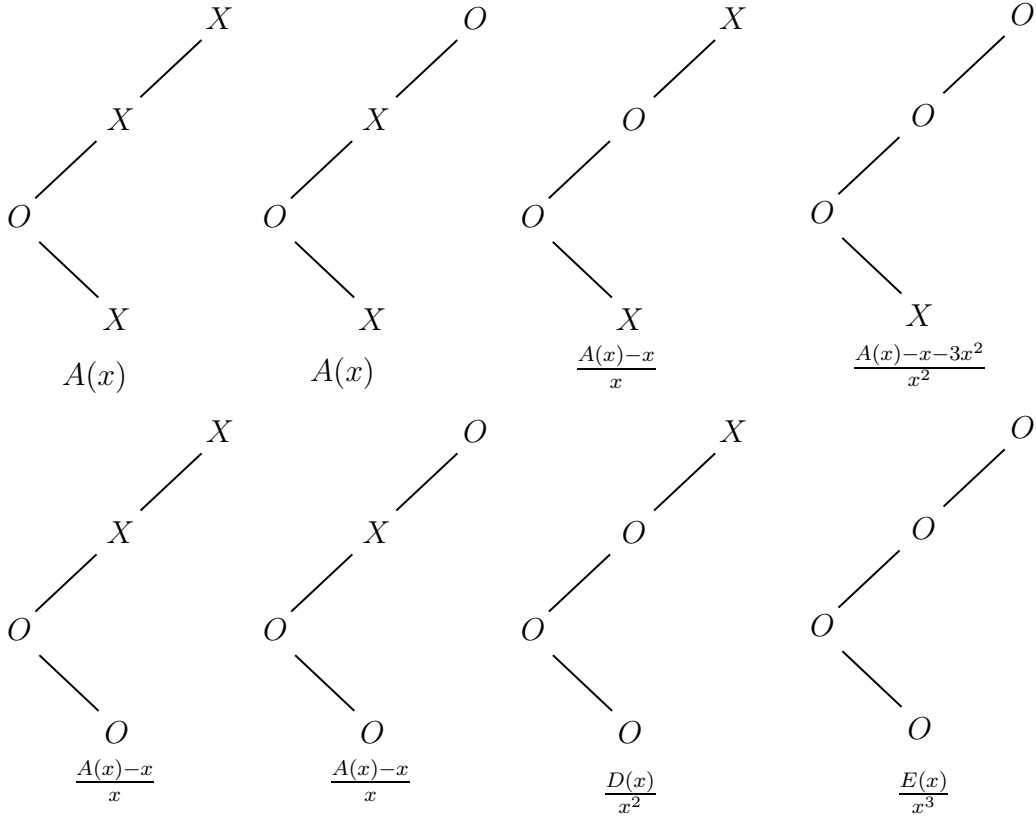
\tikzset{every picture/.style={line width=0.75pt}}
\begin{tikzpicture}[x=0.75pt,y=0.75pt,yscale=-1,xscale=1]

\draw    (53.83,82.11) -- (23,111) ;

\draw    (106.83,31.11) -- (76,60) ;

\draw    (54.83,161.11) -- (25,133) ;

\draw    (458.83,80.11) -- (428,109) ;

\draw    (511.83,29.11) -- (481,58) ;

\draw    (459.83,159.11) -- (430,131) ;

\draw    (312.83,82.11) -- (282,111) ;

\draw    (365.83,31.11) -- (335,60) ;

\draw    (313.83,161.11) -- (284,133) ;

\draw    (182.83,82.11) -- (152,111) ;

\draw    (235.83,31.11) -- (205,60) ;

\draw    (183.83,161.11) -- (154,133) ;

\draw (66,71) node   {$X$};
\draw (15,120) node   {$X$};
\draw (64,172) node   {$X$};
\draw (116,20) node   {$X$};
\draw (471,69) node   {$O$};
\draw (420,118) node   {$X$};
\draw (469,170) node   {$X$};
\draw (521,18) node   {$O$};
\draw (325,71) node   {$O$};
\draw (274,120) node   {$X$};
\draw (323,172) node   {$X$};
\draw (375,20) node   {$X$};
\draw (195,71) node   {$X$};
\draw (144,120) node   {$X$};
\draw (193,172) node   {$X$};
\draw (245,20) node   {$O$};
\draw (53,201) node   {$xA_{Disc.}( x)$};
\draw (178,199) node   {$xA_{Disc.}( x)$};
\draw (316,196) node   {$xA_{Disc.}( x)$};
\draw (469,195) node   {$3xA_{Disc.}( x)$};

\end{tikzpicture}

\begin{tikzpicture}[x=0.75pt,y=0.75pt,yscale=-1,xscale=1]

\draw    (53.83,82.11) -- (23,111) ;

\draw    (106.83,31.11) -- (76,60) ;

\draw    (54.83,161.11) -- (25,133) ;

\draw    (458.83,80.11) -- (428,109) ;

\draw    (511.83,29.11) -- (481,58) ;

\draw    (459.83,159.11) -- (430,131) ;

\draw    (312.83,82.11) -- (282,111) ;

\draw    (365.83,31.11) -- (335,60) ;

\draw    (313.83,161.11) -- (284,133) ;

\draw    (182.83,82.11) -- (152,111) ;

\draw    (235.83,31.11) -- (205,60) ;

\draw    (183.83,161.11) -- (154,133) ;

\draw (66,71) node   {$X$};
\draw (15,120) node   {$X$};
\draw (64,172) node   {$O$};
\draw (116,20) node   {$X$};
\draw (471,69) node   {$O$};
\draw (420,118) node   {$X$};
\draw (469,170) node   {$O$};
\draw (521,18) node   {$O$};
\draw (325,71) node   {$O$};
\draw (274,120) node   {$X$};
\draw (323,172) node   {$O$};
\draw (375,20) node   {$X$};
\draw (195,71) node   {$X$};
\draw (144,120) node   {$X$};
\draw (193,172) node   {$O$};
\draw (245,20) node   {$O$};
\draw (53,201) node   {$xA_{Disc.}( x)$};
\draw (178,199) node   {$xA_{Disc.}( x)$};
\draw (316,196) node   {$xA_{Disc.}( x)$};
\draw (469,195) node   {$3xA_{Disc.}( x)$};

\end{tikzpicture}

\begin{tikzpicture}[x=0.75pt,y=0.75pt,yscale=-1,xscale=1]

\draw    (51.83,82.11) -- (21,111) ;

\draw    (104.83,31.11) -- (74,60) ;

\draw    (52.83,161.11) -- (23,133) ;

\draw    (456.83,80.11) -- (426,109) ;

\draw    (509.83,29.11) -- (479,58) ;

\draw    (457.83,159.11) -- (428,131) ;

\draw    (310.83,82.11) -- (280,111) ;

\draw    (363.83,31.11) -- (333,60) ;

\draw    (311.83,161.11) -- (282,133) ;

\draw    (180.83,82.11) -- (150,111) ;

\draw    (233.83,31.11) -- (203,60) ;

\draw    (181.83,161.11) -- (152,133) ;

\draw (64,71) node   {$X$};
\draw (13,120) node   {$O$};
\draw (62,172) node   {$X$};
\draw (114,20) node   {$X$};
\draw (469,69) node   {$O$};
\draw (418,118) node   {$O$};
\draw (467,170) node   {$X$};
\draw (519,18) node   {$O$};
\draw (323,71) node   {$O$};
\draw (272,120) node   {$O$};
\draw (321,172) node   {$X$};
\draw (373,20) node   {$X$};
\draw (193,71) node   {$X$};
\draw (142,120) node   {$O$};
\draw (191,172) node   {$X$};
\draw (243,20) node   {$O$};
\draw (51,201) node   {$A( x)$};
\draw (176,199) node   {$A( x)$};
\draw (314,196) node   {$\frac{A( x) -x}{x}$};
\draw (467,195) node   {$\frac{A( x) -x-3x^{2}}{x^2}$};

\end{tikzpicture}

\begin{tikzpicture}[x=0.75pt,y=0.75pt,yscale=-1,xscale=1]

\draw    (53.83,82.11) -- (23,111) ;

\draw    (106.83,31.11) -- (76,60) ;

\draw    (54.83,161.11) -- (25,133) ;

\draw    (458.83,80.11) -- (428,109) ;

\draw    (511.83,29.11) -- (481,58) ;

\draw    (459.83,159.11) -- (430,131) ;

\draw    (312.83,82.11) -- (282,111) ;

\draw    (365.83,31.11) -- (335,60) ;

\draw    (313.83,161.11) -- (284,133) ;

\draw    (182.83,82.11) -- (152,111) ;

\draw    (235.83,31.11) -- (205,60) ;

\draw    (183.83,161.11) -- (154,133) ;

\draw (66,71) node   {$X$};
\draw (15,120) node   {$O$};
\draw (64,172) node   {$O$};
\draw (116,20) node   {$X$};
\draw (471,69) node   {$O$};
\draw (420,118) node   {$O$};
\draw (469,170) node   {$O$};
\draw (521,18) node   {$O$};
\draw (325,71) node   {$O$};
\draw (274,120) node   {$O$};
\draw (323,172) node   {$O$};
\draw (375,20) node   {$X$};
\draw (195,71) node   {$X$};
\draw (144,120) node   {$O$};
\draw (193,172) node   {$O$};
\draw (245,20) node   {$O$};
\draw (192,196) node   {$\frac{A( x) -x}{x}$};
\draw (63,196) node     {$\frac{A( x) -x}{x}$};
\draw (323,202) node   {$\frac{D( x)}{x^2}$};
\draw (471,202) node   {$\frac{E( x)}{x^3}$};

\end{tikzpicture}

Summing yields:
$$E_{Disc.}(x) = 12x^4A_{Disc.}(x) + x^3\left(1+A_{Disc.}(x)\right)\left(2A(x) + 3\frac{A(x)-x}{x} + \frac{A(x) - x- 3x^2}{x^2} + \frac{D(x)}{x^2} + \frac{E(x)}{x^3}\right)$$
$$= \frac{P_{Disc.}(x) + Q_{Disc.}(x)\sqrt{1-4x}}{R_{Disc.}(x)}.$$
Where \tiny
$$P_{Disc.}(x) = 1536 x^{16}-9586 x^{15}+20762 x^{14}-19750 x^{13}+10942 x^{12}-15139 x^{11}+27760 x^{10}-28954 x^9$$
$$+16898 x^8-4690 x^7-173 x^6+689 x^5-454 x^4+238 x^3-73 x^2+8 x$$
$$Q_{Disc.}(x) = 3224 x^{16}-13230 x^{15}+21016 x^{14}-15930 x^{13}+4800 x^{12}+3759 x^{11}-10616 x^{10}+13958 x^9$$
$$-10482 x^8+4200 x^7-695 x^6-95 x^5+168 x^4-140 x^3+57 x^2-8 x$$
\normalsize
and
$$R_{Disc.}(x) = (-1 + x)^4 (-1 + 6 x - 8 x^2 + 4 x^3).$$
\end{proof}

\section{Rationally Smooth Schubert Varieties}
In this section we enumerate the number of rationally smooth Schubert varieties in the classical finite type $E$. For any simple Lie group $G$ and any Borel subgroup $B$ take the Weyl group of $G$, $W$.
We can index the Schubert varieties $X(w)$ in the flag variety $G / B$ by elements $w \in W$. When $W = E_n$ for some $n$ we can answer how many rationally smooth Schubert varieties there are.

First we introduce the notion of a Billey-Postnikov (BP) decomposition. Let $S$ be the set of vertices in the Dynkin diagram of $W$, that is, the generators $W$ as a Coxeter group. Let $J \subset S$ and $W_J$ be the subgroup generated by $J$.
\begin{definition}\label{d:sphere}
We say a staircase $\mcD$ is spherical if $W_B$ is finite for all $B \in D$.
\end{definition}
We let $W^{J}$ denote the set of minimal length left cosets representatives of the quotient $W/W_{J}$.
Every element $w \in W$ can be uniquely decomposed as $w = vu$ with $v \in W^{J}$ and $u \in W_{J}$. This is known as the parabolic decomposition of $w$ with respect to $J$.
\begin{definition}
We say a parabolic decomposition $w = vu$ is a Billey-Postnikov decomposition if $S(v) \cap J \subseteq \mcD_{L}(u)$ where $\mcD_{L}(u)$ is the left descent set of $u$, that is $\mcD_{L}(u) = \{s \in S | l(su) = l(u) - 1\}$.
\end{definition}
Furthermore we define a complete Billey-Postnikov decomposition as follows:
\begin{definition}
$w$ has a complete BP decomposition, $w = v_{n}\cdots v_{1}$, if there is a increasing sequence of sets $\emptyset \subsetneq J_1 \subsetneq J_2 \cdots J_k = S(w)$ with $k = |S(w)|$ with $S(v_j) \subset J_{j}$, $v_j \in W^{J_{j}}$ and $v_{j}(v_{j-1} \cdots v_{1})$ is a BP decomposition.
\end{definition}
We now use a theorem from \cite{RS15a}
showing a bijection between rationally smooth Schubert varieties and complete BP decompositions.
\begin{theorem}
A Schubert variety $X(w)$ is smooth if and only if $w$ has a complete BP decomposition $w = v_k\cdots v_1$ where $X^{J_{j -1}}(v_j)$ is smooth.
\end{theorem}
Inducing the following fiber bundle structure on $X(w)$ given the complete BP decomposition $w = v_k\cdots v_1$:

\begin{center}

\tikzset{every picture/.style={line width=0.75pt}} %

\begin{tikzpicture}[x=0.75pt,y=0.75pt,yscale=-1,xscale=1]
\draw    (166.5,194) -- (166.99,266) ;
\draw [shift={(167,268)}, rotate = 269.61] [color={rgb, 255:red, 0; green, 0; blue, 0 }  ][line width=0.75]    (10.93,-3.29) .. controls (6.95,-1.4) and (3.31,-0.3) .. (0,0) .. controls (3.31,0.3) and (6.95,1.4) .. (10.93,3.29)   ;
\draw    (251.5,195) -- (251.99,267) ;
\draw [shift={(252,269)}, rotate = 269.61] [color={rgb, 255:red, 0; green, 0; blue, 0 }  ][line width=0.75]    (10.93,-3.29) .. controls (6.95,-1.4) and (3.31,-0.3) .. (0,0) .. controls (3.31,0.3) and (6.95,1.4) .. (10.93,3.29)   ;
\draw    (556.5,194) -- (556.99,266) ;
\draw [shift={(557,268)}, rotate = 269.61] [color={rgb, 255:red, 0; green, 0; blue, 0 }  ][line width=0.75]    (10.93,-3.29) .. controls (6.95,-1.4) and (3.31,-0.3) .. (0,0) .. controls (3.31,0.3) and (6.95,1.4) .. (10.93,3.29)   ;
\draw    (415.5,194) -- (415.99,266) ;
\draw [shift={(416,268)}, rotate = 269.61] [color={rgb, 255:red, 0; green, 0; blue, 0 }  ][line width=0.75]    (10.93,-3.29) .. controls (6.95,-1.4) and (3.31,-0.3) .. (0,0) .. controls (3.31,0.3) and (6.95,1.4) .. (10.93,3.29)   ;
\draw    (195.5,184) -- (216.5,184) ;
\draw [shift={(218.5,184)}, rotate = 180] [color={rgb, 255:red, 0; green, 0; blue, 0 }  ][line width=0.75]    (10.93,-3.29) .. controls (6.95,-1.4) and (3.31,-0.3) .. (0,0) .. controls (3.31,0.3) and (6.95,1.4) .. (10.93,3.29)   ;
\draw    (284.5,185) -- (306.5,185) ;
\draw [shift={(308.5,185)}, rotate = 180] [color={rgb, 255:red, 0; green, 0; blue, 0 }  ][line width=0.75]    (10.93,-3.29) .. controls (6.95,-1.4) and (3.31,-0.3) .. (0,0) .. controls (3.31,0.3) and (6.95,1.4) .. (10.93,3.29)   ;
\draw    (453.5,186) -- (512.5,185.03) ;
\draw [shift={(514.5,185)}, rotate = 539.06] [color={rgb, 255:red, 0; green, 0; blue, 0 }  ][line width=0.75]    (10.93,-3.29) .. controls (6.95,-1.4) and (3.31,-0.3) .. (0,0) .. controls (3.31,0.3) and (6.95,1.4) .. (10.93,3.29)   ;
\draw    (332.5,185) -- (359.5,185) ;
\draw [shift={(361.5,185)}, rotate = 180] [color={rgb, 255:red, 0; green, 0; blue, 0 }  ][line width=0.75]    (10.93,-3.29) .. controls (6.95,-1.4) and (3.31,-0.3) .. (0,0) .. controls (3.31,0.3) and (6.95,1.4) .. (10.93,3.29)   ;

\draw (148,173.4) node [anchor=north west][inner sep=0.75pt]    {$X( v_{1})$};
\draw (221,174.4) node [anchor=north west][inner sep=0.75pt]    {$X( v_{2} v_{1})$};
\draw (364,174.4) node [anchor=north west][inner sep=0.75pt]    {$X( v_{k-1} ...v_{1})$};
\draw (519,175.4) node [anchor=north west][inner sep=0.75pt]    {$X( v_{k} \ ...v_{1})$};
\draw (147,273.4) node [anchor=north west][inner sep=0.75pt]    {$X( v_{1})$};
\draw (227,272.4) node [anchor=north west][inner sep=0.75pt]    {$X^{J_{1}}( v_{2})$};
\draw (391,275.4) node [anchor=north west][inner sep=0.75pt]    {$X^{J_{k-2}}( v_{k-1})$};
\draw (521,274.4) node [anchor=north west][inner sep=0.75pt]    {$X^{J_{k-1}}( v_{k})$};
\draw (314,176.4) node [anchor=north west][inner sep=0.75pt]    {$...$};

\end{tikzpicture}
\end{center}
\newpage
We also have the following theorem.
\begin{theorem}
The subgroup $W_J$ is finite if $J$ is a finite type Dynkin diagram.
\end{theorem}
Which give the following corollary due to \ref{d:sphere}.
\begin{corollary}
A staircase is spherical if each block is a finite type Dynkin diagram.
\end{corollary}
Now from \cite{RS15} we have the following bijection between spherical staircase diagrams and elements $w \in W$ with complete BP decompositions.
\begin{theorem}
There is a bijection between spherical staircase diagrams and elements $w \in W$ with complete BP decomposition $w = v_{k} \cdots v_{1}$ with $v_j$ the largest element in $W_{S(v_j)} \cap W^{J}$ for all $j \in \{1 , \cdots, n\}$.
\end{theorem}

\newpage

\printbibliography
\end{document}